\documentclass[sn-mathphys-num]{sn-jnl}% Math and Physical Sciences Numbered Reference Style 
\usepackage{graphicx}%
\usepackage{multirow}%
\usepackage{amsmath,amssymb,amsfonts}%
\usepackage{amsthm}%
\usepackage{mathrsfs}%
\usepackage[title]{appendix}%
\usepackage{xcolor}%
\usepackage{textcomp}%
\usepackage{manyfoot}%
\usepackage{booktabs}%
\usepackage{algorithm}%
\usepackage{algorithmicx}%
\usepackage{algpseudocode}%
\usepackage{listings}%

\usepackage{multicol}

\newcommand\raisepunct[1]{\,\mathpunct{\raisebox{0.5ex}{#1}}}
\usepackage{url}
\usepackage{booktabs}
\theoremstyle{thmstyleone}%
\newtheorem{theorem}{Theorem}%  meant for continuous numbers
\theoremstyle{thmstyletwo}%

\theoremstyle{thmstylethree}%
\newtheorem{definition}{Definition}%

\begin{document}

\title[Padovan and Perrin Hyperbolic Spinors]{Padovan and Perrin Hyperbolic Spinors}


\author[1]{\fnm{Zehra} \sur{\.{I}\c{s}bilir}}\email{zehraisbilir@duzce.edu.tr}
\equalcont{These authors contributed equally to this work.}

\author[2]{\fnm{I\c{s}{\i}l} \sur{Arda K\"osal}}\email{isil.arda@ogr.sakarya.edu.tr}
\equalcont{These authors contributed equally to this work.}

\author*[2]{\fnm{Murat} \sur{Tosun}}\email{tosun@sakarya.edu.tr}
\equalcont{These authors contributed equally to this work.}

\affil[1]{\orgdiv{Department of Mathematics}, \orgname{D\"uzce University}, \orgaddress{ \city{D\"uzce}, \postcode{81620}, \country{T\"urkiye}}}

\affil[2]{\orgdiv{Department of Mathematics}, \orgname{Sakarya University}, \orgaddress{ \city{Sakarya}, \postcode{54187}, \country{T\"urkiye}}}


\abstract{In this study, we intend to bring together Padovan and Perrin number sequences, which are one of the most popular third-order recurrence sequences, and hyperbolic spinors, which are used in several disciplines from physics to mathematics, with the help of the split quaternions. This paper especially improves the relationship between hyperbolic spinors both a physical and mathematical concept, and number theory. For this aim, we combine the hyperbolic spinors and Padovan and Perrin numbers concerning the split Padovan and Perrin quaternions, and we determine two new special recurrence sequences named Padovan and Perrin hyperbolic spinors. Then, we give Binet formulas, generating functions, exponential generating functions, Poisson generating functions, and summation formulas. Additionally, we present some matrix and determinant equations with respect to them. Then, we construct some numerical algorithms for these special number systems, as well. Further, we give an introduction for $(s,t)$-Padovan and $(s,t)$-Perrin hyperbolic spinors.}




\maketitle

\section{Introduction}
Among the innumerable concepts associated with mathematics and physics, spinors draw attention as a basic and extensively studied topic. One of the most important notions is spinors, which researchers study from mathematics to physics. Cartan introduced the notion of the spinor in 1913. In spite of the fact that the term ‘‘spinor’’ was coined by Ehrenfest in the 1920s, the intrinsic notion of a spinor is much older than that, as a spinor is a special linear structure that had been studied (and had been used in civil engineering in order to calculate statics) long before Ehrenfest used it in quantum theory \cite{Vaz}. Spinors are a quite important concept for quantum mechanics and, therefore, modern physics as a whole. Nearly at the starting point of quantum theory, in 1927, physicists Pauli and Dirac obtained spinors to express wave function, with the former for three-dimensional space and the latter for four-dimensional space-time. When studying the representation of groups in mathematics, Cartan found spinors and explained how spinors supply a linear depiction of a space's rotations in any dimension. Because of the fact that spinors are closely related to geometry, their presentation is commonly abstract and without any clear geometric interpretation \cite{Hladik}. The spinors in a geometrical sense are examined by Cartan \cite{Cartan}. Thanks to the study \cite{Cartan}, the set of isotropic vectors of the vector space $\mathbb{C}^3$ establishes a two-dimensional surface in the two-dimensional complex space $\mathbb{C}^2$. Conversely, these vectors in $\mathbb{C}^2$ represent the same isotropic vectors. Cartan expressed that these vectors are complex as two-dimensional in the space $\mathbb{C}^2$ \cite{Cartan, erisir1}. In addition to these, the triads of unit vectors are orthogonal by twos and were expressed in terms of a single vector that has two complex components, which is called a spinor \cite{Cartan,Castillobook, delcastillo}. To get for more detailed information with respect to the spinors, we refer to the studies \cite{Vaz,Cartan,Castillobook, delcastillo, Brauer,Lounesto}. 

del Castillo and Barrales gave the spinor representations of the Frenet-Serret frame in \cite{delcastillo}, it is a milestone for several researchers in order to construct the representations of moving frames associated with the spinors.
Also, the spinor representations of involute-evolute curves \cite{erisir2}, Bertrand curves \cite{erisir1}, and successor curves \cite{erisiryeni} were examined. Then, Do\u{g}an Yaz{\i}c{\i}  et al. \cite{bahar} introduced the spinor representation of Mannheim framed (which can have singular points) curves and \.{I}\c{s}bilir et al. \cite{zehra} determined the spinor representation of Bertrand framed curves. Then, \.{I}\c{s}bilir et al. \cite{is} determined the spinor representation of framed curves in 3-dimensional Lie groups.

Moreover, the hyperbolic spinors, which are another type of spinor, were scrutinized and combined with the different and several frames. Balc\i et al. determined the hyperbolic spinor representation of space-like curves associated with the Darboux frame in Minkowski space \cite{Balci}. Eri\c{s}ir et al. introduced the hyperbolic spinor equations of alternative frame \cite{erisir3}. Also, Ketenci et al. investigated the spinor equations of curves
in Minkowski space \cite{ketenci1}, and gave the construction of hyperbolic
spinors related to Frenet frame in Minkowski space \cite{ketenci}. 

Number theory is a well-established and important topic for lots of disciplines, such as; computer systems, engineering, architecture, and others. Several studies have been undertaken and are still in progress with respect to numbers and number systems. 
Quaternions are among the most important concepts in the number system. Quaternions were determined by Hamilton in 1843 to extend the complex numbers, and the quaternion algebra is associative, non-commutative, and 4-dimensional Clifford algebra. The set of quaternions (real/Hamilton type) is denoted by $\mathbb{H}$ and defined as $\mathbb{H}=\{q\,\,|\,\,q = {q_0} + {q_1}i + {q_2}j + {q_3}k, \,\, q_0,q_1,q_2,q_3\in \mathbb{R}\}$, where $i,j,k$ are quaternionic units that satisfy the multiplication rules \cite{Hamilton,Hamilton2,Hamilton3}:
\begin{equation}\label{quaternion}
\begin{array}{cc}
i^2 = {j^2} = {k^2} =  - 1, \\
ij =  - ji = k, \quad jk =  - kj =  i, \quad ki =  - ik =  j.
\end{array}
\end{equation}
The literature includes many works concerning quaternions and lots of researchers examined them and investigated the other types of quaternions. Additionally, Cockle investigated the split quaternions \cite{JCockle}, and split quaternions satisfy the following rules \cite{JCockle,Diskaya2}:
\begin{equation*}\label{splitunits}
\begin{array}{cc}
    i^2=-1, \quad  j^2=k^2=1, \quad ijk=1, \\  ij=-ji=k, \quad jk=-kj=-i, \quad  ki=-ik=j.
    \end{array}
\end{equation*}
On the other hand, special recurrence sequences are one of the most attractive concepts for researchers. Several types of research are completed and ongoing with respect to them. If one examines the literature, it can be seen that there are lots of special sequences with different orders. In this study, we intend to examine the third-order recurrence sequences named as Padovan and Perrin numbers, which are the special types of generalized Tribonacci numbers \cite{Morales,Soykanrst}.
The Padovan sequence is a recurrence sequence of integers, and the $n^{th}$ Padovan number is represented by $P_n$ and satisfies the following recurrence relation:
\begin{equation}\label{padovanrecurrence}
{P_{n+3 }} = {P_{n +1}}+{P_{n }} \quad \text{for all} \quad n \in \mathbb{N},
\end{equation}
with the initial values $P_0=P_1=P_2=1$ \cite{soykan2}. Also, the Perrin sequence is a recurrence sequence of integers, and the $n^{th}$ Perrin number is represented by $P_n$ and satisfies the following recurrence relation:
\begin{equation}\label{perrinrecurrence}
{R_{n+3 }} = {R_{n +1}}+  {R_{n }} \quad \text{for all} \quad  n \in \mathbb{N},
\end{equation}
with the initial values $R_0=3,R_1=0,R_2=2$ \cite{soykan2}.

Padovan and Perrin numbers are studied intensively by many researchers. While the ratio of two successive Fibonacci numbers
converges to the \textit{golden ratio} \cite{Dunlap}, and the ratio of two successive Padovan numbers converges to the \textit{plastic ratio} (cf. $\alpha$ in the equation \eqref{roots}) \cite{shannon,soykan2,RichardPadovan,padovan2}.  Fibonacci and Lucas numbers are special number sequences of the second-order, and one of the concepts that makes these sequences special is the golden ratio. Padovan and Perrin numbers are third-order number sequences. Similar to the concept of golden ratio in Fibonacci numbers, there is the concept of the plastic ratio in Padovan and Perrin numbers. These number sequences have been and are being studied with quaternions in many studies. In this study, we study Padovan and Perrin numbers together with hyperbolic spinors. 

In addition to these, several researchers bring together the special recurrence sequences and special types quaternions in the existing literature and in progress. Cerda-Morales determined the real quaternions with generalized Tribonacci numbers components in \cite{Morales}. Also, Ta\c{s}c{\i} defined the real quaternions with Padovan and Pell-Padovan numbers components in \cite{TasciPadovanquaternion}. G\"unay and Ta\c{s}kara studied some properties of Padovan real quaternions in \cite{Gunay}. Then, generalized Padovan sequence, which is a subgroup of the generalized Tribonacci family, was taken and real-type Padovan, Perrin, and Van der Laan quaternions were determined in \cite{gunaytez}. Also, Di\c{s}kaya and Menken defined the real quaternions with $(s,t)$-Padovan and $(s,t)$-Perrin numbers components \cite{Diskaya1}, and split quaternions with $(s,t)$-Padovan and $(s,t)$-Perrin numbers components \cite{Diskaya2}. The $n^{th}$ split $(s,t)$-Padovan and $(s,t)$-Perrin quaternion are denoted by $\breve{P}_n$ and $\breve{R}_n$ and also determined as follows:
\begin{equation}
    \breve{P}_n=P_n+P_{n+1}i+P_{n+2}j+P_{n+3}k
\end{equation}
and
\begin{equation}
    \breve{R}_n=R_n+R_{n+1}i+R_{n+2}j+R_{n+3}k,
\end{equation}
where $P_n$ and $R_n$ are the $n^{th}$ Padovan and Perrin numbers and $i,j,k$ are the split quaternionic units that satisfy the multiplication rules given in the Equation \eqref{splitunits}. 
Then, the following recurrence relation is expressed for split $(s,t)$-Padovan and $(s,t)$-Perrin quaternions:
\begin{equation}
\breve{P}_{n+3}=s\breve{P}_{n}+t\breve{P}_{n-1} \quad \text{for all} \quad n\ge0.
\end{equation}
One can see that if $s=t=1$, then split quaternions with Padovan and Perrin numbers components are obtained \cite{Diskaya2}. The conjugate of the $n^{th}$ $(s,t)$-Padovan and $(s,t)$-Perrin split quaternion is denoted by $\breve{P}^*_{n}=P_{n}- P_{n+1}i- P_{n+2}j- P_{n+3}k$ and $\breve{R}^*_{n}=R_{n}- R_{n+1}i- R_{n+2}j- R_{n+3}k$, respectively \cite{Diskaya2,splitfibonacci}.
Also, Di\c{s}kaya and Menken \cite{Diskaya2} determined some properties, formulas, and equations with respect to $(s,t)$-Padovan and $(s,t)$-Perrin split quaternions.
Additionally, generalized quaternions with Padovan and Perrin numbers components were given by \.{I}\c{s}bilir and G\"urses in \cite{IsbilirandGurses1}. 

On the other hand, Vivarelli \cite{Vivarelli} determined the relation between the spinors and quaternions, and with respect to the relation between quaternions and rotations in Euclidean 3-space, the spinor representations of these 3-dimensional rotations were given \cite{horadamspinor}. Tarak\c{c}{\i}o\u{g}lu et al. investigated the relations between the hyperbolic spinors and split quaternions in \cite{tarakcioglu}. 
Recently, Eri\c{s}ir and G\"ung\"or introduced the Fibonacci and Lucas spinors in \cite{fibonaccispinor}, and Eri\c{s}ir determined the Horadam spinors in \cite{horadamspinor}. Moreover, Kumari et al. defined the $k$-Fibonacci and $k$-Lucas spinors in \cite{kumari}, and Leonardo spinors in \cite{kumari-2}. \"Oz\c{c}evik and Dertli determined the hyperbolic Jacobsthal spinor sequences in \cite{ozcevik}.

In this paper, we strengthen the relationships of hyperbolic spinors and special recurrence sequences by using split quaternions. For this purpose, we determine the Padovan and Perrin hyperbolic spinors and examine some properties of them. Also, we determine Binet formulas, generating functions, exponential generating functions, Poisson generating functions, and summation formulas. Moreover, we obtain some matrix and determinant equations concerning them. Additionally, we establish some numerical algorithms for these special number systems. Then, we give a short introduction for $(s,t)$-Padovan hyperbolic spinors and $(s,t)$-Perrin hyperbolic spinors. These new sequences include the Padovan and Perrin hyperbolic spinors for the values of $s$ and $t$. Consequently, we give conclusions and express our intention on how we can take this work to an even higher level in the future.

\section{Basic Concepts}
In this section, we remind some required notions and notations with respect to the used concept throughout this study such as; hyperbolic spinors, split quaternions, and Padovan and Perrin numbers.

\subsection{Hyperbolic spinors}
Suppose that $\varepsilon$ is an $n\times n$ matrix that is defined on the hyperbolic number system $\mathbb{H}$. $\varepsilon^\dagger$ defined as transposing and conjugating of $\varepsilon$, that is $\varepsilon^\dagger=\overline{\varepsilon^t}$, which is an $n\times n$ matrix. Provided that $\varepsilon$ is a Hermitian matrix with respect to $\mathbb{H}$, then $\varepsilon^t=\varepsilon$. Also, if $\varepsilon$ is an anti-Hermitian matrix with respect to $\mathbb{H}$, then $\varepsilon^t=-\varepsilon$. Let $\varepsilon$ be a Hermitian matrix, the equation $UU^\dagger=U^\dagger U=1$ is valid for $U=e^{j\varepsilon}$. The set of all $n\times n$ type matrices on $\mathbb{H}$ which satisfies the previous equation establishes a group called hyperbolic unitary group, and denoted by $U(n,\mathbb{H})$. If $\det U=1$, then this type group is represented by $SU(n,\mathbb{H})$ \cite{antonuccio,Balci,erisir3}.

Additionally, Lorentz group is a group of all Lorentz transformations in the Minkowski space and it is a subgroup of the Poincar\'e group. Moreover, Poincar\'e group is determined as the group of all isometries in the Minkowski space. The term “orthochronous” is a Lorentz transformation which is kept
in the direction of time. Then, the orthochronous Lorentz group is defined as that rigid transformation of Minkowski 3-space that kept both the direction of time and orientation. If they have the determinant $+1$, then
this subgroup is represented as $SO(1,3)$ \cite{carmeli,Balci,erisir3,ketenci}.

Further, there is a homomorphism between the group $SO(1,3)$, which is the group of the rotation along the origin, and $SU(2,\mathbb{H})$, which is the group of the unitary $2\times 2$ type matrix. While the elements of the group $SU(2,\mathbb{H})$ present a fillip to the hyperbolic spinors, the elements of the group $SO(1,3)$ give a fillip to the vectors with three real components in Minkowski space \cite{sattinger,Balci,erisir3,ketenci}.

One can represent a hyperbolic spinor with two hyperbolic components as follows:
\begin{equation*}
    \psi=\begin{pmatrix}
    \psi_1\\
    \psi_2
    \end{pmatrix}
\end{equation*}
by using the vectors $a,b,c\in\mathbb{R}_1^3$ such that
\begin{equation}\label{spinor}
\begin{split}
    a+jb&=\psi^t\sigma\psi,\\
    c&=-\widehat\psi^t\sigma\psi,
    \end{split}
\end{equation}
where “t” denotes the transposition, $\overline\psi$ is the conjugate of $\psi$, $\widehat\psi$ is the mate of $\psi$. Also, the followings can be expressed:
\begin{equation*}
   \widehat\psi=-\begin{pmatrix}
   0&1\\
   -1&0
   \end{pmatrix} \overline\psi=-\begin{pmatrix}
    0&1\\
   -1&0
   \end{pmatrix}\begin{pmatrix}
   \overline\psi_1\\
   \overline\psi_2
   \end{pmatrix}=\begin{pmatrix}
   -\overline\psi_2\\
   \overline\psi_1
   \end{pmatrix}.
\end{equation*}
Also, $2\times2$ hyperbolic symmetric matrices which are cartesian components for the vector $\varsigma=(\varsigma_1,\varsigma_2,\varsigma_3)$
\begin{equation}\label{matrices}
   \varsigma_1= \begin{pmatrix}
    1&0\\
    0&-1
    \end{pmatrix}, \quad   \varsigma_2= \begin{pmatrix}
    j&0\\
    0&j
    \end{pmatrix}, \quad  \varsigma_3= \begin{pmatrix}
    0&-1\\
    -1&0
    \end{pmatrix}
\end{equation}
are written \cite{Balci,erisir3,tarakcioglu,ketenci,ketenci1}.
The ordered triads $\{a,b,c\}, \{b,c,a\}, \{c,a,b\}$ correspond to different hyperbolic spinors, and the hyperbolic spinors $\psi$ and $-\psi$ correspond to the same ordered orthogonal basis. For the hyperbolic spinors $\psi$ and $\phi$, the following equations hold: 
\begin{align}
\psi^t\sigma\phi&=\phi^t\varsigma\psi,\label{prop2}
\\
    \overline{\psi^t\sigma\phi}&=-\widehat\psi^t\varsigma\widehat\phi,\label{prop1}
\\
\widehat{\left(\nu_1\psi+\nu_2\phi\right)}&=\overline \nu_1\widehat\psi+\overline \nu_2\widehat\phi,\label{prop3}
\end{align}
where $\nu_1,\nu_2\in\mathbb{H}$
\cite{Balci,ketenci,erisir3,ketenci1}.
Let $\xi=(\xi_1,\xi_2,\xi_3)\in\mathbb{H}^3$ be an isotropic vector (namely, length of this vector is zero: $\langle \xi,\xi \rangle=0$, $\xi\ne0$) in $\mathbb{R}_1^3$.
According to the above notions and notations, the following equations can be given:
\begin{eqnarray*}
     \xi_1=\eta_1^2-\eta^2_2, \quad  \xi_2=j(\eta^2_1+\eta^2_2), \quad \xi_3=-2\eta_1\eta_2.
\end{eqnarray*}
Also, the following equations
\begin{eqnarray}\label{xx}
     \psi_1=\pm\sqrt{\frac{\xi_1+j\xi_2}{2}} \quad \text{and} \quad \psi_2=\pm\sqrt{\frac{-\xi_1+j\xi_2}{2}}
\end{eqnarray}
can be given. In that case, $||a||=||b||=||c||=\overline\psi^t\psi$. According to the \eqref{spinor} and \eqref{matrices}, the followings
\begin{equation*}
    \begin{array}{rl}
     \xi_1=\psi^t\sigma_1\psi=\psi^2_1-\psi^2_2, \quad
       \xi_2=\psi^t\sigma_2\psi=j(\eta^2_1+\eta^2_2),\quad
        \xi_3=\psi^t\sigma_3\psi=-2\psi_1\psi_2
    \end{array}
\end{equation*}
and
\begin{align} \label{zz}
         a+jb&=\left(\psi^2_1-\psi^2_2,j(\psi^2_1+\psi^2_2),-2\psi_1\psi_2\right), \\
         c&= \left(\psi_1\overline\psi_2+\overline\psi_1\psi_2,j(\psi_1\overline\psi_2-\overline\psi_1\psi_2),\left|\psi_1\right|^2-\left|\psi_2\right|^2\right)
\end{align}
can be written \cite{Balci,ketenci,erisir3}.
For more detailed information with respect to the hyperbolic spinor (especially related to hyperbolic spinors and moving frames), we want to refer to the studies \cite{Balci,ketenci,tarakcioglu,erisir3}.

\subsection{Split quaternions}
The split quaternion $q \in \mathbb{H}_S$ can be written as $q={S_q}+{\overrightarrow V_q}$, where ${S_q} = {q_0}$ is scalar part and ${\overrightarrow V_q} = {q_1}i + {q_2}j + {q_3}k$ is vector part. For the split quaternions $q,p \in \mathbb{H}_S$ with respect to the Equation \eqref{quaternion}, some algebraic properties can be expressed as follows \cite{splitfibonacci}:
\begin{itemize}
\item \textbf{Addition/Subtraction} $$q \pm p = {q_0} \pm {p_0}+\left( {{q_1}\pm {p_1}}\right)i+\left( {{q_2} \pm {p_2}}\right)j +\left({{q_3} \pm {p_3}}\right)k,$$
\item \textbf{Multiplication by a scalar} $$\omega q = \omega{q_0} + \omega {q_1}i + \omega {q_2}j + \omega {q_3}k, \quad \omega \in \mathbb{R}
,$$
\item \textbf{Multiplication}
$$\begin{array}{rl}
qp&=\left({q_0} + {q_1}i + {q_2}j + {q_3}k\right)\left({p_0} + {p_1}i + {p_2}j + {p_3}k\right)\\
&=\left( q_0p_0-q_1p_1+q_2p_2+q_3p_3 \right)+\left( q_0p_1-q_1p_0-q_2p_3+q_3p_2 \right)i\\&\,\,\,\,+\left( q_0p_2-q_1p_3+q_2p_0+q_3p_1 \right)j+\left( q_0p_3-q_1p_2-q_2p_1+q_3p_0 \right)k\\
&=S_qS_p+g\left( \overrightarrow{V}_q,\overrightarrow{V}_p \right)+S_q\overrightarrow{V}_p+S_p\overrightarrow{V}_q+\overrightarrow{V}_q\wedge\overrightarrow{V}_p
.
\end{array}$$
where \begin{equation*}
    g\left( \overrightarrow{V}_q,\overrightarrow{V}_p\right)=-q_0p_0+q_1p_1+q_2p_2+q_3p_3
    \end{equation*}
and
\begin{equation*}
\overrightarrow{V}_q\wedge\overrightarrow{V}_p=\begin{vmatrix} \begin{array}{ccc}
    -i & j & k \\
    q_1&q_2&q_3\\
    p_1&p_2&p_3
\end{array}   \end{vmatrix}.
\end{equation*}
\item \textbf{Conjugate} The conjugate of the split quaternion $q$ is $q^*={q_0}-{q_1}i-{q_2}j-{q_3}k$.
\item \textbf{Norm:}  The norm of $q$ is: ${N_q}={qq^*}=q_0^2+q_1^2-q_2^2-q_3^2$.
\end{itemize}

\subsection{Relations between the hyperbolic spinors and split quaternions}
In \cite{tarakcioglu,tarakcioglutez}, the relations between the hyperbolic spinors and split quaternions are examined. Let the quaternion $q \in \mathbb{R}$ and the hyperbolic spinor $\psi$ be given, then we have \cite{tarakcioglu,tarakcioglutez}:
\begin{equation}\label{3}
    \begin{split}
    f:&\,\,\mathbb{H}\rightarrow\mathbb{S}\\
        &\,\,q\rightarrow f\left(q_0+q_1i+q_2j+q_3k  \right)=\begin{bmatrix}
        q_{0}+q_{3}j\\
        -q_{1}+q_{2}j
        \end{bmatrix}\equiv \psi_n,
    \end{split}
\end{equation}
where the function $f$ is linear, one-to-one, and onto. Hence, $f\left(q+p\right)=f\left(q\right)+f\left(p\right)$ and $f\left(\omega q\right)=\omega f\left(q\right)$, where $\omega\in\mathbb{R}$ and $ker f=\{0\}$. According to the conjugation of the split quaternion $q$, the following is satisfied \cite{tarakcioglu,tarakcioglutez}:
\begin{equation}
    f(q^*)= f\left(q_0-q_1i-q_2j-q_3 k \right)=\begin{bmatrix}
        q_{0}-jq_{3}\\
        q_{1}-jq_{2}
        \end{bmatrix}\equiv \psi^*_n.
\end{equation}
For more detailed information with respect to the relations and representations between the hyperbolic spinors and split quaternions, see \cite{tarakcioglu,tarakcioglutez}.

\subsection{Padovan and Perrin numbers}
The characteristic equation of the Padovan and Perrin numbers is ${x^3}-x-1=0$ and the roots of it are as follows \cite{soykan2}:
\scriptsize
\begin{equation}\label{roots}
\begin{split}
\left\{\begin{split}
{\alpha} &= \small \sqrt[3]{{\frac{1}{2} + \frac{1}{6}\sqrt {\frac{{23}}{3}} }} + \sqrt[3]{{\frac{1}{2} - \frac{1}{6}\sqrt {\frac{{23}}{3}} }}\approx 1.3247\ldots,\\
{\beta}& =\small  - \sqrt[3]{{\frac{1}{{16}} + \frac{1}{{48}}\sqrt {\frac{{23}}{3}} }} - \sqrt[3]{{\frac{1}{{16}} - \frac{1}{{48}}\sqrt {\frac{{23}}{3}} }} + \frac{{i\sqrt 3 }}{2}\left( {\sqrt[3]{{\frac{1}{2} + \frac{1}{6}\sqrt {\frac{{23}}{3}} }} - \sqrt[3]{{\frac{1}{2} - \frac{1}{6}\sqrt {\frac{{23}}{3}} }}} \right),\\
{\gamma}& = \small  - \sqrt[3]{{\frac{1}{{16}} + \frac{1}{{48}}\sqrt {\frac{{23}}{3}} }} - \sqrt[3]{{\frac{1}{{16}} - \frac{1}{{48}}\sqrt {\frac{{23}}{3}} }} - \frac{{i\sqrt 3 }}{2}\left( {\sqrt[3]{{\frac{1}{2} + \frac{1}{6}\sqrt {\frac{{23}}{3}} }} - \sqrt[3]{{\frac{1}{2} - \frac{1}{6}\sqrt {\frac{{23}}{3}} }}} \right),
\end{split}\right.
\end{split}
\end{equation}
\normalsize
where ${\alpha}+{\beta}+{\gamma}=0$, $\alpha \beta+\alpha \gamma+\beta \gamma=-1$, and ${\alpha}{\beta}{\gamma}=1$. Also, the ratio of two successive Padovan or Perrin numbers converges to the \textit{plastic ratio} (see $\alpha$ in the above Equation \eqref{roots}) \cite{shannon,soykan2,RichardPadovan,padovan2}.
In addition to these, for all $n \in \mathbb{N}$, Binet formulas of Padovan and Perrin numbers are given as follows, respectively \cite{moralesres,Diskaya2,NazmiyeTez}:
\begin{align}
 {P_n}& = \sigma_1\alpha^n + \sigma_2\beta^n + \sigma_3\gamma^n,\label{padovanbinet}\\
 {R_n} &= \alpha^n + \beta^n + \gamma^n,\label{perrinbinet}
\end{align}
where
\begin{equation*}\label{abcvalues}
\begin{split}
\sigma_1 = \frac{{({\beta} - 1)({\gamma} - 1)}}{{({\alpha} - {\beta})({\alpha} - {\gamma})}},\quad
\sigma_2 = \frac{{({\alpha} - 1)({\gamma} - 1)}}{{({\beta} - {\alpha})({\beta} - {\gamma})}},\quad 
\sigma_3 = \frac{{({\alpha} - 1)({\beta} - 1)}}{{({\gamma} - {\alpha})({\gamma} - {\beta})}}\raisepunct{.}
\end{split}
\end{equation*}
Then, for all $n\in\mathbb{N}$, the following relations between the Padovan and Perrin numbers hold \cite{NazmiyeTez,nazmiye2}:
\begin{align}
{R_n} &= 3{P_{n - 5}} + 2{P_{n - 4}},\label{naz1}\\
{P_{n - 1}}&= \frac{1}{{23}}\left({R_{n - 3}} + 8{R_{n - 2}} + 10{R_{n - 1}}\right).\label{naz3}
\end{align}
Also, Sokhuma determined the Padovan $Q$-matrix, and give some relations \cite{Sokhuma,S}:
\begin{equation*}
\begin{bmatrix}
0&1&0\\
0&0&1\\
1&1&0
\end{bmatrix}
\end{equation*}
Also, the matrix representations and matrix sequences of Padovan and Perrin numbers are given in \cite{Mangueira, NazmiyeTez,nazmiye2}.
To examine for detailed information associated with the Padovan and Perrin numbers, we can refer to the studies \cite{Morales, Diskaya1,Diskaya2, gunaytez,Gunay,IsbilirandGurses1,Kalman,Shannon,ShannonHoradamandAnderson,Soykanrst,Soykansumformulas,soykan2,onlineansiklopedi,S,Sokhuma,shannon,stewart, TasciPadovanquaternion,Waddill, WaddillandSacks,Stewart2,Elucas,Perrin1}.

\section{Padovan and Perrin Hyperbolic Spinors}
In this section, we investigate and examine new number systems bringing together the hyperbolic spinors and one of the most popular third-order special recurrence numbers Padovan and Perrin numbers with the help of the split Padovan and Perrin quaternions. Moreover, we give some algebraic properties and equalities concerning conjugations. Then, we construct some equations such as; recurrence relation, Binet formula, generating function, exponential generating function, Poisson generating function, summation formulas, and matrix formulas. Then we obtain the determinant equality for calculating the terms of these sequences. Also, we give some numerical algorithms with respect to these special number systems.

\begin{definition}\label{generalized Tribonaccispinor}
Let $\breve{P}_n$ and $\breve{R}_n$ be the $n^{th}$ Padovan split quaternion and $n^{th}$ Perrin split quaternion, respectively. The set of $n^{th}$ Padovan and Perrin split quaternion are denoted by $\mathbb{\breve{P}}$ and $\mathbb{\breve{R}}$, respectively. We can construct the following transformations with the help of the correspondence between the split quaternions and hyperbolic spinors as follows:
\begin{equation}\label{3.1}
    \begin{split}
    f:\mathbb{\breve{P}}&\rightarrow\mathbb{S}\\
        \breve{P}_n&\rightarrow f\left(P_n+P_{n+1}i+P_{n+2}j+P_{n+3}k \right)=\begin{bmatrix}
        P_{n}+P_{n+3}j\\
        -P_{n+1}+P_{n+2}j
        \end{bmatrix}\equiv \psi_n
    \end{split}
\end{equation}
and
\begin{equation}\label{3.2}
    \begin{split}
    f:\mathbb{\breve{R}}&\rightarrow\mathbb{S}\\
        \breve{P}_n&\rightarrow f\left(R_n+R_{n+1}i+R_{n+2}j+R_{n+3}k \right)=\begin{bmatrix}
        R_{n}+R_{n+3}j\\
        -R_{n+1}+R_{n+2}j
        \end{bmatrix}\equiv \phi_n,
    \end{split}
\end{equation}
where the split quaternionic units $i,j,k$ are satisfied by the rules which are given in the Equation \eqref{splitunits}. Since these transformations are linear and one-to-one but not onto, these new type sequences which are called Padovan and Perrin hyperbolic spinor sequences, respectively. These are linear recurrence sequences and are constructed by using this transformation. 
\end{definition}
Now, we give some algebraic properties with respect to the Padovan and Perrin hyperbolic spinor sequences such as; addition/subtraction and multiplication by a scalar, respectively. For the sake of brevity, we give algebraic properties only for the Padovan hyperbolic spinor since similar properties can be written easily by substituting $P$ to $R$ for Perrin hyperbolic spinors. Let us consider $\psi_n,\psi_m\in\mathbb{S}$ for $n,m\ge0$:
\begin{itemize}
    \item \textbf{Addition/Subtraction:} 
    \begin{align*}
        \psi_n \pm \psi_m = &\begin{bmatrix}
        P_{n}+P_{n+3}j\\
        -P_{n+1}+P_{n+2}j
        \end{bmatrix}\pm\begin{bmatrix}
        P_{m}+P_{n+3}j\\
        -P_{m+1}+P_{m+2}j
        \end{bmatrix}\\
        =&\begin{bmatrix}
        P_{n}\pm P_{m}+\left(P_{n+3}\pm P_{m+3}    \right)j\\
       -\left( P_{n+1}\pm P_{m+1}\right)+\left(P_{n+2}\pm P_{m+2}\right)j
        \end{bmatrix}.
        \end{align*}
    \item \textbf{Multiplication by a scalar:} 
    \begin{align*}
        \lambda\psi_n = \begin{bmatrix}
        \lambda P_{n}+\lambda P_{n+3}j\\
        -\lambda P_{n+1}+\lambda P_{n+2}j
        \end{bmatrix}, \quad \lambda\in \mathbb{R}.
    \end{align*}
    \end{itemize}
    It should be noted that throughout this paper, we use the following explanations in order to the sake of brevity:
    \begin{itemize}
        \item $\breve{P}_n$ is the $n^{th}$ Padovan split quaternion, $\breve{R}_n$ is $n^{th}$ Perrin split quaternion, $\mathbb{\breve{P}}$ is the set of Padovan split quaternions, $\mathbb{\breve{R}}$ is the set of Perrin split quaternion, $\psi_n$ is the $n^{th}$ Padovan hyperbolic spinor, $\phi_n$ is the $n^{th}$ Perrin hyperbolic spinor. 
    \end{itemize}
\begin{theorem}[\textbf{Recurrence Relation}]
For all $n\ge0$, the following recurrence relations are given for the Padovan and Perrin hyperbolic spinor sequence, respectively:
\begin{equation}\label{recurrencerelation}
    \psi_{n+3}=\psi_{n+1}+\psi_{n},
\end{equation}
\begin{equation}\label{recurrencerelation-1}
    \phi_{n+3}=\phi_{n+1}+\phi_{n}.
\end{equation}
\end{theorem}
\begin{proof} By using the Equation \eqref{padovanrecurrence} and Equation \eqref{3.1}, we completed the proof of Equation \eqref{recurrencerelation} as follows:
    \begin{equation*}
    \begin{split}
        \psi_{n+1}+\psi_{n}=&\begin{bmatrix}
         P_{n+1}+ P_{n+4}j\\
        - P_{n+2}+ P_{n+3}j
        \end{bmatrix}+\begin{bmatrix}
         P_{n}+ P_{n+3}j\\
        - P_{n+1}+ P_{n+2}j
        \end{bmatrix}\\
       = &\begin{bmatrix}
         P_{n+1}+ P_{n+4}j+ P_{n}+ P_{n+3}j\\
        - P_{n+2}+ P_{n+3}j- P_{n+1}+ P_{n+2}j
        \end{bmatrix}\\
= &\begin{bmatrix}
         P_{n+1}+P_{n}+ \left(P_{n+4}+ P_{n+3}\right)j\\
        - \left(P_{n+2}+ P_{n+1}\right)+ \left(P_{n+3}+ P_{n+2}\right)j
        \end{bmatrix}\\
        = &\begin{bmatrix}
         P_{n+3}+P_{n+6}j\\
       - P_{n+4}+P_{n+5}j
        \end{bmatrix}\\
    =&\psi_{n+3}.
     \end{split}
\end{equation*}
The Equation \eqref{recurrencerelation-1} can be proved by using the Equation \eqref{perrinrecurrence} and Equation \eqref{3.2}.
\end{proof}
The following initial values are written for Padovan and Perrin hyperbolic spinors, respectively:
\begin{equation}\label{values}
        \psi_0 = \begin{bmatrix}
       1+2j\\
        -1+j
        \end{bmatrix},\quad 
        \psi_1 = \begin{bmatrix}
        1+2j\\
       -1+2j
        \end{bmatrix},\quad 
        \psi_2 = \begin{bmatrix}
        1+3j\\
    -2+2j
        \end{bmatrix},
    \end{equation}
and
    \begin{equation}\label{values-2}
        \phi_0 = \begin{bmatrix}
       3+3j\\
        2j
        \end{bmatrix},\quad 
        \phi_1 = \begin{bmatrix}
        2j\\
       -2+3j
        \end{bmatrix},\quad 
        \phi_2 = \begin{bmatrix}
        2+5j\\
    -3+2j
        \end{bmatrix}.
    \end{equation}

In the following Table 1 and Table 2, we construct a numerical algorithm for calculating the Padovan and Perrin hyperbolic spinors by using the recurrence relations Equation \eqref{recurrencerelation} and Equation \eqref{recurrencerelation-1}.
\begin{table}[!ht]
\centering
\caption{A numerical algorithm for finding ${n^{th}}$ term of the Padovan hyperbolic spinor}
\begin{tabular}{| l |}
  \hline 	 
  \textbf{Numerical Algorithm} 	\\ \hline	
{\bf{1.}} Begin\\
{\bf{2.}} Input $\psi_0,\psi_1$ and $\psi_2$\\
{\bf{3.}} Form $\psi_n$ according to the Equation \eqref{recurrencerelation}  \\
{\bf{4.}} Calculate $\psi_n$   \\
{\bf{5.}} Output  $ \psi_n\equiv\begin{bmatrix}
        P_{n}+P_{n+3}j\\
        -P_{n+1}+P_{n+2}j
        \end{bmatrix}$	\\
        {\bf{6.}} Complete
\\ \hline  
\end{tabular}
\end{table}
\begin{table}[!ht]
\centering
\caption{A numerical algorithm for finding ${n^{th}}$ term of the Perrin hyperbolic spinor}\label{tab2-1}
\begin{tabular}{| l |}
  \hline 	 
  \textbf{Numerical Algorithm} 	\\ \hline	
{\bf{1.}} Begin\\
{\bf{2.}} Input $\phi_0,\phi_1$ and $\phi_2$\\
{\bf{3.}} Form $\phi_n$ according to the Equation \eqref{recurrencerelation-1}  \\
{\bf{4.}} Calculate $\phi_n$   \\
{\bf{5.}} Output $ \phi_n\equiv\begin{bmatrix}
        R_{n}+R_{n+3}j\\
        -R_{n+1}+R_{n+2}j
        \end{bmatrix}$	\\
        {\bf{6.}} Complete
\\ \hline  
\end{tabular}
\end{table}

\newpage

\begin{definition} Let the conjugate of the $n^{th}$ Padovan and Perrin split quaternion is denoted by $\breve{P}^*_{n}=P_{n}- P_{n+1}i- P_{n+2}j- P_{n+3}k$ and $\breve{R}^*_{n}=R_{n}- R_{n+1}i- R_{n+2}j- R_{n+3}k$. The following expressions can be written:
\begin{itemize}
    \item 
The $n^{th}$ Padovan hyperbolic spinor $\psi_n^*$ and $n^{th}$ Perrin hyperbolic spinor $\phi_n^*$ corresponding to the conjugate of the $n^{th}$ Padovan split quaternion and $n^{th}$ Perrin split quaternion are expressed by, respectively:
\begin{equation*}
\begin{split}
f\left(\breve{P}_{n}^{*}\right)&=f\left(P_{n}-P_{n+1}i-P_{n+2}j- P_{n+3}k\right)=\left[\begin{array}{l}
P_{n}-P_{n+3}j  \\
P_{n+1}-P_{n+2}j
\end{array}\right] \equiv \psi_{n}^{*},\\
f\left(\breve{R}_{n}^{*}\right)&=f\left(R_{n}-R_{n+1}i-R_{n+2}j- R_{n+3}k\right)=\left[\begin{array}{l}
R_{n}-R_{n+3}j  \\
R_{n+1}-R_{n+2}j
\end{array}\right] \equiv \phi_{n}^{*}.
\end{split}
\end{equation*}
\item 
Also, the matrix
$C=\left[\begin{array}{cc}0 & 1 \\
-1 & 0\end{array}\right]$ is given. 
The ordinary hyperbolic conjugate of $n^{th}$ Padovan and Perrin hyperbolic spinor $\psi_n$ and $\phi_n$ is written as follows:
\begin{equation*}
\begin{split}
\overline{\psi}_n&=\left[\begin{array}{c}
P_n-P_{n+3}j \\
-P_{n+1}- P_{n+2}j
\end{array}\right],\\
\overline{\phi}_n&=\left[\begin{array}{c}
R_n-R_{n+3}j  \\
-R_{n+1}- R_{n+2}j
\end{array}\right].
\end{split}
\end{equation*}
\item
Hyperbolic conjugate of Padovan and Perrin hyperbolic spinor $\tilde{\psi}_{n}=jC \overline{\psi}_{n}$ and $\tilde{\phi}_{n}=jC \overline{\phi}_{n}$ of $n^{th}$ Padovan and Perrin hyperbolic spinor $\psi_{n}$ and $\phi_{n}$ are expressed as:
\begin{equation*}
\begin{split}
\tilde{\psi}_n&=\left[\begin{array}{l}
-P_{n+2}- P_{n+1}j \\
P_{n+3}- P_nj
\end{array}\right],\\
\tilde{\phi}_n&=\left[\begin{array}{l}
-R_{n+2}- R_{n+1}j \\
R_{n+3}- R_nj
\end{array}\right],
\end{split}
\end{equation*}
where by using the study of Cartan \cite{Cartan}.
\item 
Additionally, the hyperbolic mate of $n^{th}$ Padovan and Perrin hyperbolic spinor \linebreak  $\check{\psi}_{n}=-C \overline\psi_{n}$ and $\check{\phi}_{n}=-C \overline\phi_{n}$ are
\begin{equation*}
\begin{split}
\check{\psi}_n&=\left[\begin{array}{l}
P_{n+1}+ P_{n+2}j \\
P_n-P_{n+3}j 
\end{array}\right],\\
\check{\phi}_n&=\left[\begin{array}{l}
R_{n+1}+ R_{n+2}j \\
R_n-R_{n+3}j 
\end{array}\right],
\end{split}
\end{equation*}
where by using the study of del Castillo and Barrales \cite{delcastillo}.
\end{itemize}
\end{definition}
The following Theorem \ref{th-2}-Theorem \ref{th-8} is given without proofs, since the proofs are clear, by using the matrix $C$ and the conjugation properties of Padovan and Perrin hyperbolic spinors. 
\begin{theorem}\label{th-2}
The following equations can be expressed:
\begin{multicols}{3}
\begin{itemize}
\item [\normalfont{(a)}] $ 
 \overline{\psi}_n=C \check{\psi}_{n}$,
\item[\textnormal{(b)}] 
$\overline{\phi}_n=C \check{\phi}_{n},
$
\end{itemize}
\columnbreak
\begin{itemize}
\item[\textnormal{(c)}]$ 
\,  \check{\psi}_{n}=-j \tilde{\psi}_{n}$,
\item[\textnormal{(d)}]
$\check{\phi}_{n}=-j \tilde{\phi}_{n}$,
\end{itemize}
\columnbreak
\begin{itemize}
\item[\textnormal{(e)}] $
  \overline{\psi}_{n}=-j C \tilde{\psi}_{n}$, 
\item[\textnormal{(f)}]
 $\overline{\phi}_{n}=-j C \tilde{\phi}_{n}.$
\end{itemize}
 \end{multicols}
\end{theorem}

\begin{theorem}\label{th-3}
    The following equations are satisfied:
    \begin{multicols}{2}
    \begin{itemize}
        \item[\textnormal{(a)}] $
\psi_n+\psi_n^*=\begin{bmatrix}
            2P_n\\
            0
        \end{bmatrix},$
\item[\textnormal{(b)}]
 $\phi_n+\phi_n^*=\begin{bmatrix}
            2R_n\\
            0
        \end{bmatrix},$
        \end{itemize}
        \columnbreak
        \begin{itemize}
        \item[\textnormal{(c)}] $\psi_n-\psi_n^*=2\begin{bmatrix}
            P_{n+3}j\\
            -P_{n+1}+P_{n+2}j
        \end{bmatrix}$,
\item[\textnormal{(d)}] 
$\phi_n-\phi_n^*=2\begin{bmatrix}
            R_{n+3}j\\
            -R_{n+1}+R_{n+2}j
        \end{bmatrix}.$
\end{itemize}
 \end{multicols}
\end{theorem}

\begin{theorem}\label{th-4}
The following equations are satisfied:
\begin{multicols}{2}
\begin{itemize}
 \item[\textnormal{(a)}] $
 \,   \psi_n+\overline{\psi}_n=2\begin{bmatrix}
            P_n\\
              - P_{n+1}
          \end{bmatrix}$,
\item[\textnormal{(b)}] 
$\phi_n+\overline{\phi}_n=2\begin{bmatrix}
            R_n\\
              - R_{n+1}\\
          \end{bmatrix}$,
            \end{itemize}
        \columnbreak
        \begin{itemize}
        \item[\textnormal{(c)}] $
\psi_n-\overline{\psi}_n=2j\begin{bmatrix}
              P_{n+3}\\
              P_{n+2}
          \end{bmatrix}$, 
\item[\textnormal{(d)}] 
$\phi_n-\overline{\phi}_n=2j\begin{bmatrix}
        R_{n+3}\\
             R_{n+2}
          \end{bmatrix}$.
\end{itemize}
\end{multicols}
\end{theorem}

\begin{theorem}\label{th-5-1}
The following equations hold:
    \begin{itemize}
    \item[\textnormal{(a)}]  $
 \psi_n+\widetilde{\psi}_n=\begin{bmatrix}
                -P_{n-1}+P_{n}j\\
                P_{n}+P_{n-1}j
            \end{bmatrix}$, 
\item[\textnormal{(b)}]
$\phi_n+\widetilde{\phi}_n=\begin{bmatrix}
                -R_{n-1}+R_{n}j\\
                R_{n}+R_{n-1}j
            \end{bmatrix}$,
\item [\textnormal{(c)}] $  \psi_n+\widetilde{\psi}_n=\begin{bmatrix}
              P_{n}+P_{n+2}+\left(P_{n+3}+P_{n+1}\right)j\\
               -P_{n+1}-P_{n+3}+\left(P_{n+2}+P_{n}   \right)j
           \end{bmatrix}$,
\item[\textnormal{(d)}]
$\phi_n+\widetilde{\phi}_n=\begin{bmatrix}
               R_{n}+R_{n+2}+\left(R_{n+3}+R_{n+1}\right)j\\
               -R_{n+1}-R_{n+3}+\left(R_{n+2}+R_{n}   \right)j
           \end{bmatrix}$.
    
        \end{itemize}
\end{theorem}

\begin{theorem}\label{th-5}
The following equations are satisfied:
    \begin{itemize}
\item [\textnormal{(a)}] $  \psi_n+\check{\psi}_n=\begin{bmatrix}
               P_{n+3}+P_{n+5}j\\
               -P_{n+1}+P_n+\left(P_{n+2}-P_{n+3}   \right)j
           \end{bmatrix}$,
\item[\textnormal{(b)}]
$\phi_n+\check{\phi}_n=\begin{bmatrix}
               R_{n+3}+R_{n+5}j\\
               -R_{n+1}+R_n+\left(R_{n+2}-R_{n+3}   \right)j
           \end{bmatrix}$,
    \item[\textnormal{(c)}]  $
 \psi_n-\check{\psi}_n=\begin{bmatrix}
                P_{n}-P_{n+1}+\left(P_{n+3}-P_{n+2}\right)j\\
                -P_{n+3}+P_{n+5}j
            \end{bmatrix}$, 
\item[\textnormal{(d)}]
$\phi_n-\check{\phi}_n=\begin{bmatrix}
                R_{n}-R_{n+1}+\left(R_{n+3}-R_{n+2}\right)j\\
                -R_{n+3}+R_{n+5}j
            \end{bmatrix}$.
        \end{itemize}
\end{theorem}

\begin{theorem}\label{th-6}
The following equations are satisfied:
\begin{multicols}{2}
\begin{itemize}
        \item [\textnormal{(a)}] $ \psi_n^*+\overline{\psi}_n=2\begin{bmatrix}
                P_{n}-P_{n+3}j\\
                P_{n+2}j
            \end{bmatrix}$,
\item[\textnormal{(b)}]
$\phi_n^*+\overline{\phi}_n=2\begin{bmatrix}
                R_{n}-R_{n+3}j\\
                R_{n+2}j
            \end{bmatrix},
        $
        \end{itemize}
        \columnbreak
        \begin{itemize}
\item[\textnormal{(c)}]  $\psi_n^*-\overline{\psi}_n=2\begin{bmatrix}
                0\\
                P_{n+1}
            \end{bmatrix}$,
\item[\textnormal{(d)}] $\phi_n^*-\overline{\phi}_n=2\begin{bmatrix}
                0\\
                R_{n+1}
            \end{bmatrix}.
        $
        \end{itemize}
        \end{multicols}
\end{theorem}

\begin{theorem}\label{th-6-1}
The following equations are satisfied:
\begin{itemize}
        \item [\textnormal{(a)}] $  \psi_n^*+\widetilde{\psi}_n=\begin{bmatrix}
                P_{n}-P_{n+2}-\left(P_{n+3}+P_{n+1}\right)j\\
                P_{n+1}+P_{n+3}-\left(P_{n+2}+P_{n}\right)j
            \end{bmatrix}$,
 \item [\textnormal{(b)}] $  \phi_n^*+\widetilde{\phi}_n=\begin{bmatrix}
                R_{n}-R_{n+2}-\left(R_{n+3}+R_{n+1}\right)j\\
                R_{n+1}+R_{n+3}-\left(R_{n+2}+R_{n}\right)j
            \end{bmatrix}$,
        \item [\textnormal{(c)}] $  \psi_n^*-\widetilde{\psi}_n=\begin{bmatrix}
                P_{n}+P_{n+2}-\left(P_{n+3}-P_{n+1}\right)j\\
                P_{n+1}-P_{n+3}-\left(P_{n+2}-P_{n}\right)j
            \end{bmatrix}$,
 \item [\textnormal{(d)}] $  \phi_n^*+\widetilde{\phi}_n=\begin{bmatrix}
                R_{n}+R_{n+2}-\left(R_{n+3}-R_{n+1}\right)j\\
                R_{n+1}-R_{n+3}-\left(R_{n+2}-R_{n}\right)j
            \end{bmatrix}$.
        \end{itemize}
\end{theorem}

\begin{theorem}\label{th-7}
The following equations are satisfied:
    \begin{itemize}
                        \item [\textnormal{(a)}]       $ \psi_n^*+\check{\psi}_n=\begin{bmatrix}
                P_{n+3}+P_{n+5}j\\
                P_{n+3}-P_{n+5}j
            \end{bmatrix}$,
\item[\textnormal{(b)}] $\phi_n^*+\check{\phi}_n=\begin{bmatrix}
                R_{n+3}+R_{n+5}j\\
                R_{n+3}-R_{n+5}j
            \end{bmatrix},$    
\item[\textnormal{(c)}]       $   \psi_n^*-\check{\psi}_n=\begin{bmatrix}
                P_{n}-P_{n+1}-\left(P_{n+3}+P_{n+2}\right)j\\
               P_{n+1}-P_{n}-\left(P_{n+3}-P_{n+2}\right)j
            \end{bmatrix}$, 
\item[\textnormal{(d)}] 
$\phi_n^*-\check{\phi}_n=\begin{bmatrix}
                R_{n}-R_{n+1}-\left(R_{n+3}+R_{n+2}\right)j\\
               R_{n+1}-R_{n}-\left(R_{n+3}-R_{n+2}\right)j
            \end{bmatrix}.$
\end{itemize}
\end{theorem}

\begin{theorem}\label{th-8-1}
The following equations are satisfied:
\begin{itemize}
\item[\textnormal{(a)}]$
\overline{\psi}_n+\widetilde{\psi}_n=\begin{bmatrix}
                -P_{n-1}-\left(P_{n+3}+P_{n+1}\right)j\\
               -P_{n}-\left(P_{n+2}+P_n\right)j
            \end{bmatrix}$,
\item[\textnormal{(b)}]$
\overline{\phi}_n+\widetilde{\phi}_n=\begin{bmatrix}
                -R_{n-1}-\left(R_{n+3}+R_{n+1}\right)j\\
               -R_{n}-\left(R_{n+2}+R_n\right)j
            \end{bmatrix}$,
 \item  [\textnormal{(c)}] $
   \overline{\psi}_n+\widetilde{\psi}_n=\begin{bmatrix}
                P_{n}+P_{n+2}-P_{n}j\\
               -P_{n+1}-P_{n+3}-P_{n-1}j
            \end{bmatrix}$, 
 \item  [\textnormal{(d)}] $
   \overline{\phi}_n+\widetilde{\phi}_n=\begin{bmatrix}
                R_{n}+R_{n+2}-R_{n}j\\
               -R_{n+1}-R_{n+3}-R_{n-1}j
            \end{bmatrix}$, 
\end{itemize}
\end{theorem}

\begin{theorem}\label{th-8-2}
The following equations are satisfied:
\begin{itemize}
\item[\textnormal{(a)}]$
\widetilde{\psi}_n+\check{\psi}_n=\begin{bmatrix}
                P_{n+1}-P_{n+2}+\left(P_{n+2}-P_{n+1}\right)j\\
               P_{n+3}+P_n-P_{n+1}j
            \end{bmatrix}$,
\item[\textnormal{(b)}]$
\widetilde{\phi}_n+\check{\phi}_n=\begin{bmatrix}
                R_{n+1}-R_{n+2}+\left(R_{n+2}-R_{n+1}\right)j\\
               R_{n+3}+R_n-R_{n+1}j
            \end{bmatrix}$,
 \item  [\textnormal{(c)}] $
   \widetilde{\psi}_n+\check{\psi}_n=\begin{bmatrix}
                -P_{n+4}-P_{n+4}j\\
               P_{n+1}-P_{n+1}j
            \end{bmatrix}$, 
 \item  [\textnormal{(d)}] $
   \widetilde{\phi}_n+\check{\phi}_n=\begin{bmatrix}
                -R_{n+4}-R_{n+4}j\\
               R_{n+1}-R_{n+1}j
            \end{bmatrix}$,
\end{itemize}
\end{theorem}

\begin{theorem}\label{th-8}
The following equations are satisfied:
\begin{itemize}
\item[\textnormal{(a)}]$
 \overline{\psi}_n-\check{\psi}_n=\begin{bmatrix}
                -P_{n+4}-P_{n+4}j\\
               P_{n+1}-P_{n+1}j
            \end{bmatrix}$, 
\item[\textnormal{(b)}]    $\overline{\phi}_n-\check{\phi}_n=\begin{bmatrix}
                -R_{n+4}-R_{n+4}j\\
               R_{n+1}-R_{n+1}j
            \end{bmatrix},$
 \item  [\textnormal{(c)}] $
   \overline{\psi}_n+\check{\psi}_n=\begin{bmatrix}
                -P_{n+2}+P_{n+1}+\left(P_{n+2}-P_{n+1}\right)j\\
               P_{n+3}+P_{n}-\left(P_{n}+P_{n+3}\right)j
            \end{bmatrix}$, 
\item[\textnormal{(d)}] $\overline{\phi}_n+\check{\phi}_n=\begin{bmatrix}
                -R_{n+2}+R_{n+1}+\left(R_{n+2}-R_{n+1}\right)j\\
               R_{n+3}+R_{n}-\left(R_{n}+R_{n+3}\right)j
            \end{bmatrix}.$
\end{itemize}
\end{theorem}

\begin{definition} [\textbf{Norm}] The norm of the $n^{th}$ Padovan and Perrin split quaternions $N\left(\breve{P}_n\right)=\breve{P}_n\breve{P}_n^*$ and $N\left(\breve{R}_n\right)=\breve{R}_n\breve{R}_n^*$ are equal to the norm of the associated Padovan and Perrin hyperbolic spinors:
\begin{equation*}
N(\breve{P}_n)=\overline\psi_n^t\psi_n,
\end{equation*}
\begin{equation*}
N(\breve{R}_n)=\overline\phi_n^t\phi_n.
\end{equation*}
and also by using the Theorem \ref{th-2}, we can express the norm of the Padovan and Perrin hyperbolic spinors as follows, respectively:
\begin{equation*}
N\left(\psi_n\right)=\overline{\psi}_n^{t}\psi_n,
\end{equation*}
\begin{equation*}
N\left(\phi_n\right)=\overline{\phi}_n^{t}\phi_n.
\end{equation*}
\end{definition}

\begin{theorem}[\textbf{Generating Function}]
For all $n\ge0$, the following generating functions hold for the Padovan and Perrin hyperbolic spinors as follows, respectively:
\begin{align}
    \sum\limits_{n = 0}^\infty {{{\psi }_n}}&=\cfrac{1}{1-x^2-x^3}\begin{bmatrix}
     1+x+(2+2x+x^2)j\\
        -1-x-x^2+(1+2x+x^2)j
    \end{bmatrix}, \label{aa}
    \\
     \sum\limits_{n = 0}^\infty {{{\phi }_n}}&=\cfrac{1}{1-x^2-x^3}\begin{bmatrix}
     3-x^2+(3+2x+2x^2)j\\
        -2x-3x^2+(2+3x)j
    \end{bmatrix}.\label{bb}
\end{align}
\end{theorem}
\begin{proof}
Suppose that the following equation is the generating function for Padovan hyperbolic spinors:
\begin{equation}\label{q}
 \sum\limits_{n = 0}^\infty {\psi}_n {x^n}  = {{ \psi}_0} + {{\psi}_1}x + {\psi_2}{x^2} + \ldots + {{ \psi}_n}{x^n} + \ldots
 \end{equation}
 Then, multiplying the Equation \eqref{q} by $x^2$ and $x^3$ and by using the Equation \eqref{recurrencerelation}, we get:
 \begin{equation*}
 \sum\limits_{n = 0}^\infty {{{\psi }_n}}-x^2\sum\limits_{n = 0}^\infty {{{\psi }_n}}-x^3\sum\limits_{n = 0}^\infty {{{\psi }_n}}=\psi_0+\psi_1x+\left(\psi_2-\psi_0 \right)x^2.
 \end{equation*}
With the help of the initial values of Padovan hyperbolic spinor sequence $\psi_0,\psi_1$ and $\psi_2$ given in the Equation \eqref{values}, then we attain:
 \begin{equation*}
 \begin{split}
 \sum\limits_{n = 0}^\infty {{{\psi }_n}}-x^2\sum\limits_{n = 0}^\infty {{{\psi }_n}}-x^3\sum\limits_{n = 0}^\infty {{{\psi }_n}}&=\begin{bmatrix}
       1+2j\\
        -1+j
        \end{bmatrix}+\begin{bmatrix}
        1+2j\\
        -1+2j
        \end{bmatrix}x+\begin{bmatrix}
        j\\
        -1+j
        \end{bmatrix}x^2\\
        &=\begin{bmatrix}
       1+x+(2+2x+x^2)j\\
        -1-x-x^2+(1+2x+x^2)j
        \end{bmatrix}.
        \end{split}
 \end{equation*}
 Therefore, we obtain:
 \begin{equation*}
\sum\limits_{n = 0}^\infty {{{\psi }_n}}=\cfrac{1}{1-x^2-x^3}\begin{bmatrix}
    1+x+(2+2x+x^2)j\\
        -1-x-x^2+(1+2x+x^2)j
    \end{bmatrix}.
\end{equation*}
Then, Equation \eqref{bb} can be proved by using the same manner.
\end{proof}

\begin{theorem}[\textbf{Binet Formula}] For all $n\ge0$, the following Binet formulas hold for Padovan and Perrin hyperbolic spinors, respectively: 
\begin{align}
\psi_n&=A\sigma_1\alpha^n+B\sigma_2\beta^n+C\sigma_3\gamma^n,\label{padovanspinorbinet}
\\
\phi_n&=A\alpha^n+B\beta^n+C\gamma^n,\label{perrinspinorbinet}
\end{align}
where
\begin{eqnarray*}
A=\begin{bmatrix}
        1+\alpha^{3}j\\
       \alpha\left(-1+\alpha j\right)
        \end{bmatrix}, \quad 
B=\begin{bmatrix}
         1+\beta^{3}j\\
       \beta\left(-1+\beta j\right)
        \end{bmatrix},\quad
C=\begin{bmatrix}
        1+\gamma^{3}j\\
       \gamma\left(-1+\gamma j\right)
        \end{bmatrix}.
\end{eqnarray*}
\end{theorem}

\begin{proof}
    By the Equation \eqref{3.1} and Binet formula of the Padovan numbers given in the Equation \eqref{padovanbinet}, we write:
    \begin{equation*}
\begin{split}
   \psi_n=&\begin{bmatrix}
\sigma_1\alpha^n+\sigma_2\beta^n+\sigma_3\gamma^n+\left(\sigma_1\alpha^{n+3}+\sigma_2\beta^{n+3}+\sigma_3\gamma^{n+3}  \right)j\\
         -\left(\sigma_1\alpha^{n+1}+\sigma_2\beta^{n+1}+\sigma_3\gamma^{n+1}\right)+\left(\sigma_1\alpha^{n+2}+\sigma_2\beta^{n+2}+\sigma_3\gamma^{n+2}  \right)j
        \end{bmatrix}\vspace{1mm}\\
      =&  \begin{bmatrix}
        1+\alpha^3j\\
       \alpha\left(-1+\alpha j\right)
        \end{bmatrix}\sigma_1\alpha^n
    +\begin{bmatrix}
        1+\beta^3j\\
       \beta\left(-1+\beta j\right)
        \end{bmatrix}\sigma_2\beta^n
        +\begin{bmatrix}
        1+\gamma^3 j\\
       \gamma\left(-1+\gamma j\right)
        \end{bmatrix}\sigma_3\gamma^n
        \vspace{1mm}
        \\
        =&A\sigma_1\alpha^n
    +B\sigma_2\beta^n
        +C\sigma_3\gamma^n\raisepunct{.}
        \end{split}
\end{equation*}
Therefore, we completed the proof of the Binet formula for Padovan numbers. Also, by using a similar way, the Equation \eqref{perrinspinorbinet} can be proved easily.
\end{proof}

\begin{theorem}[\textbf{Exponential Generating Function}]
The exponential generating functions are satisfied for Padovan and Perrin hyperbolic spinors, respectively:
 \begin{align}
   \sum\limits_{n = 0}^\infty \psi_n \cfrac{y^n}{n!} & = A\sigma_1e^{\alpha y}+B\sigma_2e^{\beta y}+C\sigma_3e^{\gamma y},\label{egfforpadovandq}\\
  \sum\limits_{n = 0}^\infty \phi_n \cfrac{y^n}{n!} & = Ae^{\alpha y}+Be^{\beta y}+Ce^{\gamma y},\label{egfforperrindq}
 \end{align}
\end{theorem}
\begin{proof}
With the help of the Equation \eqref{padovanspinorbinet}, we have the followings:
\begin{equation*}
\begin{split}
     \sum\limits_{n = 0}^\infty \psi_n \frac{y^n}{n!} &= \sum\limits_{n = 0}^\infty  \left( A\sigma_1\alpha^n+B\sigma_2\beta^n+C\sigma_3\gamma^n
  \right) \frac{y^n}{n!}\\
     &= A\sigma_1 \sum\limits_{n = 0}^\infty \alpha^n\frac{y^n}{n!}+ B\sigma_2 \sum\limits_{n = 0}^\infty \beta^n\frac{y^n}{n!}+  C\sigma_3 \sum\limits_{n = 0}^\infty \gamma^n \frac{y^n}{n!}\\
      &= A\sigma_1 \sum\limits_{n = 0}^\infty \frac{(\alpha y)^n}{n!}+ B\sigma_2 \sum\limits_{n = 0}^\infty \frac{(\beta y)^n}{n!}+  C\sigma_3 \sum\limits_{n = 0}^\infty \frac{(\gamma y)^n}{n!}\\
     &=A\sigma_1e^{\alpha_y}+B\sigma_2e^{\beta y}+C\sigma_3e^{\gamma y}.
     \end{split}
\end{equation*}
The Equation \eqref{egfforperrindq} can be demonstrated easily.
\end{proof}

\begin{theorem}[\textbf{Poisson Generating Function}]
   The following Poisson generating functions hold for Padovan and Perrin hyperbolic spinors, respectively:
    \begin{align}
   e^{-y}\sum\limits_{n = 0}^\infty \psi_n \cfrac{y^n}{n!} & = e^{-y}A\sigma_1e^{\alpha y}+e^{-y}B\sigma_2e^{\beta y}+e^{-y}C\sigma_3e^{\gamma y},\\
 e^{-y} \sum\limits_{n = 0}^\infty \phi_n \cfrac{y^n}{n!} & = e^{-y}Ae^{\alpha y}+e^{-y}Be^{\beta y}+e^{-y}Ce^{\gamma y}.
 \end{align}
\end{theorem}

\begin{proof}
    By using Equation \eqref{egfforpadovandq} and Equation \eqref{egfforperrindq}, we get the desired results, since the Poisson generating function is expressed as multiplying the exponential generating function by $e^{-y}$ (cf. also \cite{senturk-new}).
\end{proof}

\begin{theorem}
The following matrix equations are satisfied for Padovan and Perrin hyperbolic spinors, respectively:
\begin{multicols}{2}
\begin{itemize}
    \item[\textnormal{(a)}] $
  \left( {\begin{array}{*{20}{c}}
{ \psi}_{n}\\
{ \psi}_{n+1}\\
{ \psi}_{n+2}
\end{array}} \right)= \left( {\begin{array}{*{20}{c}}
0&1&0\\
0&0&1\\
1&1&0
\end{array}} \right)^n  {\left( {\begin{array}{*{20}{c}}
{\psi}_{0}\\
{\psi}_{1}\\
{\psi}_{2}
\end{array}} \right)}
$
\columnbreak
\item[\textnormal{(b)}] $
  \left( {\begin{array}{*{20}{c}}
{ \phi}_{n}\\
{ \phi}_{n+1}\\
{ \phi}_{n+2}
\end{array}} \right)= \left( {\begin{array}{*{20}{c}}
0&1&0\\
0&0&1\\
1&1&0
\end{array}} \right)^n  {\left( {\begin{array}{*{20}{c}}
{\phi}_{0}\\
{\phi}_{1}\\
{\phi}_{2}
\end{array}} \right)}
$
\end{itemize}
\end{multicols}
\end{theorem}
\begin{proof}
    By using mathematical induction, we can complete the proof. For the sake of brevity, we omit them.
\end{proof}

Thanks to the Soykan \cite{Soykansumformulas,soykan2}, we get the following summation formulas for Padovan and Perrin hyperbolic spinors without proofs.
\begin{theorem}
For all $m,n\ge0$, the following summation formulas are satisfied for Padovan and Perrin hyperbolic spinors, respectively:
\begin{itemize}
    \item[\textnormal{(a)}] $ \left\{\begin{array}{l}
\star \,  \sum\limits_{n = 0}^m {{{\psi}_n}}  = \psi_{m + 5} - \psi_4=\begin{bmatrix}
        P_{m+5}+P_{m+8}j\\
        -P_{m+6}+P_{m+7}j
        \end{bmatrix}-\begin{bmatrix}
        2+5j\\
       -3+4j
        \end{bmatrix}, \vspace{2mm}\\
        \star \,  \sum\limits_{n = 0}^m {{{\phi}_n}}  = \phi_{m + 5} - \phi_4=\begin{bmatrix}
        R_{m+5}+R_{m+8}j\\
        -R_{m+6}+R_{m+7}j
        \end{bmatrix}-\begin{bmatrix}
        2+7j\\
       -5+5j
        \end{bmatrix},
\end{array}\right.$

\item[\textnormal{(b)}] $ \left\{\begin{array}{l}
\star \,
\sum\limits_{n = 0}^m {\psi_{2n}}  = \psi_{2m + 3} - {\psi_1}=\begin{bmatrix}
        P_{2m+3}+P_{2m+6}j\\
        -P_{2m+4}+P_{2m+5}j
        \end{bmatrix}-\begin{bmatrix}
        1+2j\\
       -1+2j
        \end{bmatrix},\vspace{2mm}\\
        \star \,
\sum\limits_{n = 0}^m {\phi_{2n}}  = \phi_{2m + 3} - {\phi_1}=\begin{bmatrix}
        R_{2m+3}+R_{2m+6}j\\
        -R_{2m+4}+R_{2m+5}j
        \end{bmatrix}-\begin{bmatrix}
        2j\\
       -2+3j
        \end{bmatrix},
\end{array}\right.$

 \item[\textnormal{(c)}] $ 
 \left\{\begin{array}{l}
        \star \, \sum\limits_{n = 0}^m {{\psi}_{2n + 1}}  = \psi_{2m + 4} - \psi_2=\begin{bmatrix}
        P_{2m+2}+P_{2m+5}j\\
        -P_{2m+3}+P_{2m+4}j
        \end{bmatrix}-\begin{bmatrix}
        1+3j\\
       -2+2j
        \end{bmatrix},\vspace{2mm}\\
          \star \, \sum\limits_{n = 0}^m {{\phi}_{2n + 1}}  = \phi_{2m + 4} - \phi_2=\begin{bmatrix}
        R_{2m+4}+R_{2m+7}j\\
        -R_{2m+5}+R_{2m+6}j
        \end{bmatrix}-\begin{bmatrix}
        2+5j\\
       -3+2j
        \end{bmatrix}
        .
\end{array}\right.$
          \end{itemize}
\end{theorem}

\begin{theorem} For all $n\in \mathbb{N}$, the following relations between Padovan and Perrin hyperbolic spinors are satisfied:
\begin{itemize}
    \item [\textnormal{(a)}]
${\phi_n} = 3{\psi_{n - 5}} + 2{\psi_{n - 4}}$,
  \item [\textnormal{(b)}]
${\psi_{n - 1}}= \frac{1}{{23}}({\phi_{n - 3}} + 8{\phi_{n - 2}} + 10{\phi_{n - 1}})$.
\end{itemize}
\end{theorem}
\begin{proof}
    By using the Equations \eqref{naz1}, \eqref{3.1}, and \eqref{3.2}, we get the followings:
    \begin{equation*}
    \begin{split}
        3{\psi_{n - 5}} + 2{\psi_{n - 4}}=&
  3 \begin{bmatrix}
        P_{n-5}+P_{n-2}j\\
        -P_{n-4}+P_{n-3}j
    \end{bmatrix}+ 2\begin{bmatrix}
        P_{n-4}+P_{n-1}j\\
        -P_{n-3}+P_{n-2}j
    \end{bmatrix}\\
    =&\begin{bmatrix}
        3P_{n-5}+2P_{n-4}+\left(3P_{n-2} +2P_{n-1} \right)j\\
         -3P_{n-4}-2P_{n-3}+\left( 3P_{n-3} +2P_{n-2} \right)j\\
    \end{bmatrix}\\
    =&\begin{bmatrix}
        R_n+R_{n+3}j\\
        -R_{n+1}+R_{n+2}j
    \end{bmatrix}\\
    =&\phi_n.
    \end{split}
     \end{equation*}
     The other equation can be proved by using the Equations  \eqref{naz3}, \eqref{3.1}, and \eqref{3.2}.
\end{proof}

\begin{theorem}\label{specialdeterminatequalities}
For all $n \in \mathbb{N}$, the following determinant equations are satisfied for Padovan and Perrin hyperbolic spinors:
\begin{equation}\label{aaa}
    \psi_n=
\left| {\begin{array}{*{20}{c}}
\psi_0&-1&0&0&0&\hdots&0&0\\
\psi_1&0&-1&0&0&\hdots&0&0\\
\psi_2&0&0&-1&0&\hdots&0&0\\
0&1&1&0&-1&\hdots&0&0\\
\vdots&\ddots&\ddots&\ddots&\ddots&\ddots&\vdots&\vdots\\
0&0&0&0&0&\ddots&0&-1\\
0&0&0&0&0&\ddots&1&0
\end{array}}\right|_{(n+1)\times(n+1)},
\end{equation}
\begin{equation}\label{bbb}
   \phi_n=
\left| {\begin{array}{*{20}{c}}
\phi_0&-1&0&0&0&\hdots&0&0\\
\phi_1&0&-1&0&0&\hdots&0&0\\
\phi_2&0&0&-1&0&\hdots&0&0\\
0&1&1&0&-1&\hdots&0&0\\
\vdots&\ddots&\ddots&\ddots&\ddots&\ddots&\vdots&\vdots\\
0&0&0&0&0&\ddots&0&-1\\
0&0&0&0&0&\ddots&1&0
\end{array}}\right|_{(n+1)\times(n+1)}.
\end{equation}
\end{theorem}
Now, let us give a numerical algorithm related to the Theorem \ref{specialdeterminatequalities} in order to calculate the $n^{th}$ terms of Padovan and Perrin hyperbolic spinors in Table 3 and Table 4.
\begin{proof}
    For the sake of the brevity, we skip this proof. By using the recurrence relation of Padovan and Perrin hyperbolic spinor sequences Equations \eqref{recurrencerelation} and \eqref{recurrencerelation-1}, and the study of K{\i}z{\i}late\c{s} et al. \cite{canbicomplextribonacci} (see Theorem 5 on page 5).
\end{proof}

\begin{table}[!ht]
\centering
\caption{A numerical algorithm for finding ${n^{th}}$ term of the Padovan hyperbolic spinor}\label{tab1-1}
\begin{tabular}{| l |}
  \hline 	 
  \textbf{Numerical Algorithm} 	\\ \hline	
{\bf{1.}} Begin\\
{\bf{2.}} Input $\psi_0,\psi_1$ and $\psi_2$\\
{\bf{3.}} Form $\psi_n$ according to the Equation \eqref{aaa}  \\
{\bf{4.}} Calculate $\psi_n$   \\
{\bf{5.}} Output  $ \psi_n\equiv\begin{bmatrix}
        P_{n}+P_{n+3}j\\
        -P_{n+1}+P_{n+2}j
        \end{bmatrix}$	\\
        {\bf{6.}} Complete
\\ \hline  
\end{tabular}
\end{table}

\begin{table}[!ht]
\centering
\caption{A numerical algorithm for finding ${n^{th}}$ term of the Perrin hyperbolic spinor}\label{tab2}
\begin{tabular}{| l |}
  \hline 	 
  \textbf{Numerical Algorithm} 	\\ \hline	
{\bf{1.}} Begin\\
{\bf{2.}} Input $\phi_0,\phi_1$ and $\phi_2$\\
{\bf{3.}} Form $\phi_n$ according to the Equation \eqref{bbb}  \\
{\bf{4.}} Calculate $\phi_n$   \\
{\bf{5.}} Output  $ \phi_n\equiv\begin{bmatrix}
        R_{n}+R_{n+3}j\\
        -R_{n+1}+R_{n+2}j
        \end{bmatrix}$	\\
        {\bf{6.}} Complete
\\ \hline  
\end{tabular}
\end{table}
\newpage
Thanks to the Cereceda \cite{cerecedadet} and Cerda-Morales \cite{moralescomplex}, we get another method to calculating the $n^{th}$ Padovan and Perrin hyperbolic spinors:
\begin{theorem}
For all $n \in \mathbb{N}$, the following determinant equations hold for Padovan and Perrin hyperbolic spinors:
    \begin{equation}\label{ss}
{ \psi}_{n}=
\left| {\begin{array}{*{20}{c}}
{\psi}_0&1&0&0&\hdots&0&0\\
-{\psi}_1&0&\cfrac{1}{{\psi}_0}&0&\hdots&0&0\\
0&-{\psi}_2&0&1&\hdots&0&0\\
0&{\psi}_0&-1&0&\hdots&0&0\\
0&0&1&-1&\hdots&0&0\\
\vdots&\ddots&\ddots&\ddots&\ddots&\ddots&\vdots\\
0&0&0&0&\hdots&0&1&\\
0&0&0&0&\hdots&-1&0\\
\end{array}} \right|_{(n+1)\times(n+1)},
\end{equation}
    \begin{equation}\label{sss}
{ \phi}_{n}=
\left| {\begin{array}{*{20}{c}}
{\phi}_0&1&0&0&\hdots&0&0\\
-{\phi}_1&0&\cfrac{1}{{\phi}_0}&0&\hdots&0&0\\
0&-{\phi}_2&0&1&\hdots&0&0\\
0&{\phi}_0&-1&0&\hdots&0&0\\
0&0&1&-1&\hdots&0&0\\
\vdots&\ddots&\ddots&\ddots&\ddots&\ddots&\vdots\\
0&0&0&0&\hdots&0&1&\\
0&0&0&0&\hdots&-1&0\\
\end{array}} \right|_{(n+1)\times(n+1)}.
\end{equation}
\end{theorem}
\begin{proof}
    For the sake of brevity, we skip this proof.
\end{proof}

\begin{table}[!ht]
\centering
\caption{A numerical algorithm for finding ${n^{th}}$ term of the Padovan hyperbolic spinor}\label{tab3}
\begin{tabular}{| l |}
  \hline 	 
  \textbf{Numerical Algorithm} 	\\ \hline	
{\bf{1.}} Begin\\
{\bf{2.}} Input $\psi_0,\psi_1$ and $\psi_2$\\
{\bf{3.}} Form $\psi_n$ according to the Equation \eqref{ss}  \\
{\bf{4.}} Calculate $\psi_n$   \\
{\bf{5.}} Output  $ \psi_n\equiv\begin{bmatrix}
        P_{n}+P_{n+3}j\\
        -P_{n+1}+P_{n+2}j
        \end{bmatrix}$	\\
        {\bf{5.}} Complete
\\ \hline  
\end{tabular}
\end{table}

\begin{table}[!ht]
\centering
\caption{A numerical algorithm for finding ${n^{th}}$ term of the Perrin hyperbolic spinor}\label{tab4}
\begin{tabular}{| l |}
  \hline 	 
  \textbf{Numerical Algorithm} 	\\ \hline	
{\bf{1.}} Begin\\
{\bf{2.}} Input $\phi_0,\phi_1$ and $\phi_2$\\
{\bf{3.}} Form $\phi_n$ according to the Equation \eqref{sss}  \\
{\bf{4.}} Calculate $\phi_n$   \\
{\bf{5.}} Output  $ \phi_n\equiv\begin{bmatrix}
        R_{n}+R_{n+3}j\\
        -R_{n+1}+R_{n+2}j
        \end{bmatrix}$	\\
        {\bf{5.}} Complete
\\ \hline  
\end{tabular}
\end{table}

\newpage
\section{New Further Notions: $(s,t)$-Padovan and $(s,t)$-Perrin Hyperbolic Spinors }
In this section, we provide a short introduction to some realizable new arguments for where to go from here. These notions include the Padovan and Perrin hyperbolic spinors according to the given values $s=1$ and $t=1$.
\begin{definition}
Let $\widetilde{P}_n$ and $\widetilde{R}_n$ be the $n^{th}$ $(s,t)$-Padovan split quaternion and $n^{th}$ $(s,t)$-Perrin split quaternion. The set of $n^{th}$ $(s,t)$-Padovan and $(s,t)$-Perrin split quaternion are denoted by $\mathbb{\widetilde{P}}$ and $\mathbb{\widetilde{R}}$, respectively. We can construct the following transformations with the help of the correspondence between the split quaternions and hyperbolic spinors as follows:
\begin{equation}
    \begin{split}
    f:\mathbb{\widetilde{P}}&\rightarrow\mathbb{S}\\
        \widetilde{P}_n&\rightarrow f\left(P_n+P_{n+1}i+P_{n+2}j+P_{n+3}k \right)=\begin{bmatrix}
        P_{n}+P_{n+3}j\\
        -P_{n+1}+P_{n+2}j
        \end{bmatrix}\equiv \widetilde\psi_n
    \end{split}
\end{equation}
and
\begin{equation}
    \begin{split}
    f:\mathbb{\widetilde{R}}&\rightarrow\mathbb{S}\\
        \widetilde{P}_n&\rightarrow f\left(R_n+R_{n+1}i+R_{n+2}j+R_{n+3}k \right)=\begin{bmatrix}
        R_{n}+R_{n+3}j\\
        -R_{n+1}+R_{n+2}j
        \end{bmatrix}\equiv \widetilde\phi_n,
    \end{split}
\end{equation}
where the split quaternionic units $i,j,k$ are satisfied by the rules which are given in the Equation \eqref{splitunits}. Since these transformations are linear and one-to-one but not onto, these new type sequences which are called $(s,t)$-Padovan and $(s,t)$-Perrin hyperbolic spinor sequences, respectively. These are linear recurrence sequences and are constructed by using this transformation. 
\end{definition}

\begin{theorem}[\textbf{Recurrence Relation}]
For all $n\ge0$, the following recurrence relations are given for the $(s,t)$-Padovan and $(s,t)$-Perrin hyperbolic spinor sequence, respectively:
\begin{equation}
   \widetilde \psi_{n+3}=s\widetilde\psi_{n+1}+t\widetilde\psi_{n},
\end{equation}
\begin{equation}
    \widetilde\phi_{n+3}=s\widetilde\phi_{n+1}+t\widetilde\phi_{n}.
\end{equation}
\end{theorem}

The other equations given in the previous section, can be constructed easily for $(s,t)$-Padovan and $(s,t)$-Perrin hyperbolic spinors.

\section{Conclusion and Future Studies}
In this paper, we determined Padovan and Perrin hyperbolic spinors. Then, we obtained some algebraic properties and equalities with respect to the conjugations. Also, we gave some equations such as; recurrence relation, Binet formula, generating function, exponential generating function, Poisson generating function, summation formulas, and matrix formulas. Additionally, we got the determinant equalities for calculating the terms of these sequences. Then, we constructed some numerical algorithms concerning these special number systems. Also, we gave a short introduction with respect to the $(s,t)$-Padovan and $(s,t)$-Perrin hyperbolic spinors.

In the near future study, we intend to examine generalized Tribonacci hyperbolic spinors.

\section*{Declarations}
\bmhead{Funding} Not applicable.
\bmhead{Conflict of interest} Not applicable.
\bmhead{Ethics approval} Not applicable.
\bmhead{Consent to participate} Not applicable.
\bmhead{Consent for publication} Not applicable.
\bmhead{Availability of data and materials} Consent for publication.
\bmhead{Code availability}  Consent for publication.
\bmhead{Authors' contributions} All authors contributed equally.



\end{document}